\documentclass[11pt]{amsart}

\usepackage{amsmath,amssymb,latexsym,soul,cite}
\usepackage{xcolor,enumitem,graphicx}
\usepackage[colorlinks=true,urlcolor=blue,
citecolor=red,linkcolor=blue,linktocpage,pdfpagelabels,
bookmarksnumbered,bookmarksopen]{hyperref}
\usepackage[english]{babel}
\usepackage[T1]{fontenc}
\usepackage[utf8]{inputenc}
\usepackage{cancel}

\usepackage[left=2.8cm,right=2.8cm,top=2.9cm,bottom=2.9cm]{geometry}

\numberwithin{equation}{section}

\newtheorem{thm}{Theorem}[section]
\newtheorem*{thms}{Theorem}
  \theoremstyle{plain}
  
  \theoremstyle{plain}
  \newtheorem{prop}[thm]{Proposition}
  \theoremstyle{plain}
  
  \theoremstyle{plain}
  \newtheorem{definition}[thm]{Definition}
    \theoremstyle{definition}
\newtheorem{rem}[thm]{Remark}

\newcommand{\R}{{\mathbb R}}
\newcommand{\N}{{\mathbb N}}

\title[On Hopf's Lemma for the fractional laplacian]{On Hopf's Lemma for sign-changing supersolutions to fractional laplacian equations}

\author[A. DelaTorre]{Azahara DelaTorre}
\author[E. Parini]{Enea Parini}

\address[A. DelaTorre]{Dipartimento di Matematica Guido Castelnuovo. 
	Facoltà Scienze matematiche, fisiche e naturali.
	Sapienza Università di Roma. 
	Piazzale Aldo Moro, 5, 00185 Roma RM.} \email{azahara.delatorrepedraza@uniroma1.it}

\address[E. Parini]{Aix Marseille Univ, CNRS, I2M, 3 place Victor Hugo, 13331 Marseille CEDEX 03, France}
\email{enea.parini@univ-amu.fr}

\subjclass[2010]{}
\keywords{}

\thanks{}

\begin{document}

\begin{abstract}
In this paper we investigate the validity of Hopf's Lemma for a (possibly sign-changing) function $u \in H^s_0(\Omega)$ satisfying
\[ (-\Delta)^s u(x) \geq c(x)u(x) \quad \text{in }\Omega,\]
where $\Omega \subset \R^N$ ($N\geq 1$) is an open set, $c \in L^\infty(\Omega)$ and $(-\Delta)^s u$ is the fractional Laplacian of $u$. We show that, under suitable assumptions, the validity of Hopf's Lemma for $u$ at a point $x_0 \in \partial \Omega$ is essentially equivalent to the validity of Hopf's Lemma for the Caffarelli-Silvestre extension of $u$ at the point $(x_0,0) \in \R^N \times \R^+$. We also provide a slightly more precise characterization of a dichotomy result stated in a recent paper by Dipierro, Soave and Valdinoci \cite{dipierrosoavevaldinoci}.
\end{abstract}

\maketitle

\section{Introduction}

The classical Hopf's Lemma, proven independently by Eberhard Hopf \cite{hopf} and Olga Oleinik \cite{oleinik}, is one of the best known results in the theory of elliptic PDEs. Several variants and extension of the original result have been proven throughout the years. Here we will mention the following version.

\begin{thms}[Hopf's Lemma] Let $\Omega \subset \R^N$ be a bounded, open domain, let $c \in L^\infty(\Omega)$, and let $u \in H^1_0(\Omega)$ solve
	\[ \left\{\begin{array}{r c l l} -\Delta u & \geq & c(x)u & \text{in }\Omega, \\ u & = & 0 & \text{on }\partial \Omega.\end{array} \right.
	\]
Let $x_0 \in \partial \Omega$, and suppose that $\Omega$ satisfies an interior ball condition at $x_0$, that is, there exists a ball $B\subset \Omega$ centered at $y_0 \in \Omega$ such that $x_0 \in \partial B$. Suppose that $u$ is continuous at $x_0$, and that there exists an open neighbourhood $V$ of $x_0$ such that $u(x) > 0$ for every $x \in \Omega \cap V$. Then,
\[ \liminf_{\stackrel{x \to x_0}{x \in C_\beta(x_0)}} \frac{u(x)}{|x-x_0|} > 0,\]
where $\beta \in \left(0, \frac{\pi}{2}\right)$, and
	\[ C_\beta(x_0) := \left\{x \in \Omega \,\bigg|\,\arccos\left(\frac{x-x_0}{|x-x_0|}\cdot \frac{y_0-x_0}{|y_0-x_0|}\right) < \beta\right\}.\]
\end{thms}
A proof can be found, for instance, in \cite[Lemma 3.4]{gilbargtrudinger}. The strategy consists in finding suitable barrier functions in a subset of the ball provided by the interior ball condition, and applying the weak maximum principle.

\medskip

It is natural to wonder whether Hopf's Lemma holds true also for integrodifferential operators such as the fractional Laplacian. It turns out that, due to the nonlocal effects, the situation is much more involved than in the case of standard second-order, local differential operators. The first result in this direction was obtained by Greco and Servadei \cite{grecoservadei}, who were able to prove the validity of Hopf's Lemma for lower semicontinous supersolutions of
	\[ \left\{\begin{array}{r c l l} (-\Delta)^s u & = & c(x)u & \text{in }\Omega, \\ u & = & 0 & \text{in }\R^N \setminus \Omega.\end{array} \right.
\]
under the global sign assumption $u \geq 0$ in $\Omega$. As for sign-changing supersolutions, the validity of Hopf's Lemma was first proven by Fall and Jarohs \cite[Proposition 3.3]{falljarohs} assuming that they are antisymmetric with respect to a hyperplane. 

\medskip

More recently, a striking result by Dipierro, Soave and Valdinoci \cite{dipierrosoavevaldinoci} showed that Hopf's Lemma \emph{does not hold true}, in general, for the fractional Laplacian. Indeed, the authors were able to exhibit a counterexample, and they showed the following dichotomy:
\begin{itemize}
	\item Either $u$ ``grows faster than the power $2s$'', and Hopf's Lemma holds true;
	\item or $u$ ``does not grow faster than the power $2s$'', and Hopf's Lemma does not hold true.
\end{itemize}
The precise meaning of the growth condition will be given in Section \ref{secUu}.

\medskip

Using a different approach from that in \cite{dipierrosoavevaldinoci}, the authors of \cite{delatorreparini} provided a sufficient condition for the validity of Hopf's Lemma for sign-changing solutions at a point $x_0 \in \partial \Omega$, which is in particular satisfied when the so-called \emph{Caffarelli-Silvestre extension} (or \emph{$s$-harmonic extension}), which is the solution to a \emph{local} problem, satisfies Hopf's Lemma at $(x_0,0) \in \R^N \times \R^+$. 

\medskip

Let us now turn to the content of this paper. After introducing notation and recalling some known results in Section \ref{sec:pre}, in Section \ref{secuU} we first revisit the idea used in \cite[Proposition 4.5]{delatorreparini} to obtain the validity of Hopf's Lemma for (possibly sign-changing) supersolutions to the equation
	\begin{equation} \label{eq:supersolutionintroduction}
	\left\{\begin{array}{r c l l}
		(-\Delta)^s u &=& c(x)u& \text{in }\Omega, \\ u &= & 0 & \text{in }\R^N \setminus\Omega,\end{array}\right.
\end{equation}
where $c \in L^\infty(\Omega)$, assuming that Hopf's Lemma is satisfied by their $s$-harmonic extension. This relation was also highlighted in \cite[
Theorem 5.2]{fallweth-eig}, by means of different techniques. Observe that \cite[Proposition 4.5]{delatorreparini} deals only with the case when $\Omega$ is a ball, but the proof can be easily adapted, under suitable assumptions, to the case of a general domain. We also deal with the opposite implication; more precisely, we prove that when a supersolution to \eqref{eq:supersolutionintroduction} satisfies Hopf's Lemma at a point $x_0\in\partial\Omega$, then its $s$-harmonic extension, which is defined in $\R^N \times \R^+$, satisfies Hopf's Lemma at $(x_0,0)$. To the best of our knowledge, this implication had not been previously addressed. Finally, in Section \ref{secUu}, as an application of our results we show that the growth condition mentioned in \cite{dipierrosoavevaldinoci}, which is a sufficient condition for the validity of Hopf's Lemma, can be slightly weakened.

\section{Definitions and preliminary results}\label{sec:pre}

Let $N\geq 1$. For a function $u : \R^N \to \R$, we will denote $u^+$ and $u^-$ the functions defined in $\R^N$ by
\[ u^+(x) := \max\{u(x),0\}, \qquad u^-(x) := \max\{-u(x),0\},\]
so that $u = u^+ - u^-$.

For $s \in (0,1)$, the fractional Sobolev space $H^s(\R^N)$ is defined as
\[ H^s(\R^N) := \left\{ u \in L^2(\R^N) \,\bigg|\, \int_{\R^N} \int_{\R^N} \frac{|u(x)-u(y)|^2}{|x-y|^{N+2s}}\,dx\,dy < +\infty\right\}.\]
The quantity
\[ [u]_{H^s(\R^N)} := \left( \int_{\R^N} \int_{\R^N} \frac{|u(x)-u(y)|^2}{|x-y|^{N+2s}}\,dx\,dy \right)^\frac{1}{2}\]
is the \emph{Gagliardo seminorm} of the function $u$.

Let $\Omega \subset \R^N$ be an open set. The space $H^s_0(\Omega)$ is defined as the closure of $C^\infty_c(\Omega)$ with respect to the Gagliardo seminorm. If the boundary of $\Omega$ is sufficiently smooth (for instance, Lipschitz-continuous), it holds
\[ H^s_0(\Omega) = \{ u \in H^s(\R^N)\,|\, u = 0 \text{ quasi-everywhere in }\R^N \setminus \Omega\}. \]

The \emph{fractional Laplacian} $(-\Delta)^s u$ of a function $u \in C^2_{loc}(\R^N)$ is defined as
\[ (-\Delta)^s u(x) := C_{N,s} \lim_{\varepsilon \to 0^+} \int_{\R^N \setminus B_\varepsilon(x)} \frac{u(x)-u(y)}{|x-y|^{N+2s}}\,dy,\]
where $C_{N,s}$ is a normalizing constant defined as \[C_{N,s}= \pi^{\frac{N}{2}}2^{2s} \frac{\Gamma\left(\frac{N+2s}{2}\right)}{\Gamma(2-s)} s(1-s) > 0.\]

Let $f \in L^2(\Omega)$ and $g \in H^s(\R^N)$. We say that $u \in H^s(\R^N)$ is a \emph{weak solution} to the boundary value problem
	\begin{equation}
	\left\{\begin{array}{r c l l}
		(-\Delta)^s u &=& f & \text{in }\Omega, \\ u &= & g & \text{in }\R^N \setminus\Omega,\end{array}\right.
\end{equation}
if $u-g \in H^s_0(\Omega)$, and
\[ \int_{\R^N} \int_{\R^N} \frac{(u(x)-u(y))(\varphi(x)-\varphi(y))}{|x-y|^{N+2s}}\,dx\,dy = \int_\Omega f(x)\varphi(x)\,dx\]
for every $\varphi \in H^s_0(\Omega)$. We say that $u \in H^s(\R^N)$ is a \emph{supersolution} to the equation
\begin{equation}
	\left\{\begin{array}{r c l l}
		(-\Delta)^s u & = & f & \text{in }\Omega, \\ u &= & g & \text{in }\R^N \setminus\Omega,\end{array}\right.
\end{equation}
if $u-g \in H^s_0(\Omega)$, and
\[ \int_{\R^N} \int_{\R^N} \frac{(u(x)-u(y))(\varphi(x)-\varphi(y))}{|x-y|^{N+2s}}\,dx\,dy \geq \int_\Omega f(x)\varphi(x)\,dx\]
for every $\varphi \in H^s_0(\Omega)$ such that $\varphi \geq 0$ in $\Omega$.

The \emph{first eigenvalue} of the fractional Laplacian in $\Omega$ under homogeneous boundary conditions is the quantity $\lambda_1(\Omega)$ defined as
\[ \lambda_1(\Omega) := \inf_{u \in H^s_0(\Omega) \setminus \{0\}} \frac{[u]_{H^s(\R^N)}^2}{\|u\|_{L^2(\Omega)}^2}.\]
Using the standard method of Calculus of Variations, it is easy to prove that $\lambda_1(\Omega) > 0$ for every open set $\Omega \subset \R^N$ of finite measure. Moreover, by a simple scaling argument,
\[ \lambda_1(t \Omega) = \frac{\lambda_1(\Omega)}{t^{2s}} \qquad \text{for every }t > 0,\]
where
\[ t \Omega := \left \{ x \in \R^N \,\bigg|\, \frac{x}{t} \in \Omega\right\}.\]

Let $u \in H^s(\R^N)$. The \emph{Caffarelli-Silvestre extension} of $u$ (see \cite{caffarellisilvestre}) is a function $U: \R^{N+1}_+ \to \R$ which solves 
\begin{equation} \label{eq:caffarellisilvestre}
	\left\{\begin{array}{r c l l}
		\text{div}(t^{1-2s} \nabla U) &=& 0& \text{in }\R^{N+1}_+, \\ U(x,0) &= & u(x) & \text{in }\R^N.\end{array}\right.
\end{equation}
Here $\R^{N+1}_+ := \R^N \times \R^+$. Therefore, $U$ is a solution to a \emph{local}, degenerate PDE, which satisfies
\[U \in L^2_{loc}(\R^{N+1}_+; t^{1-2s}) \qquad \text{and} \qquad \nabla U \in L^2(\R^{N+1}_+; t^{1-2s}),\]
and is smooth in the interior of $\R^{N+1}_+$ by standard elliptic regularity (see \cite[Section 3.1]{cabresire}). Moreover, the unique solution to \eqref{eq:caffarellisilvestre} satisfies the representation formula  (see \cite[Section 2.4]{caffarellisilvestre})
\begin{equation}\label{eq:representation}
U(x,t)=P_{N,s}\int_{\R^n}\frac{u(y)t^{2s}}{(|x-y|^2+t^2)^\frac{N+2s}{2}}\,dy,\end{equation}
where $P_{N,s}=\pi^{-\frac{N}{2}}\frac{\Gamma\left(\frac{N+2s}{2}\right)}{\Gamma\left(s\right)}$ is a positive constant depending only on $N$ and $s$. 
\medskip

%
%
%
%

\section{Hopf's Lemma}\label{secuU}
In this section we are going to highlight the following fact: under suitable assumptions, a (possibly sign-changing) supersolution $u$ to the problem
	\begin{equation}
	\left\{\begin{array}{r c l l}
		(-\Delta)^s u &=& c(x)u& \text{in }\Omega, \\ u &= & 0 & \text{in }\R^N \setminus\Omega,\end{array}\right.
\end{equation}
where $c \in L^\infty(\Omega)$, satisfies Hopf's Lemma at a point $x_0 \in \partial \Omega$ if and only if the Caffarelli-Silvestre extension of $u$, defined as the solution $U$ to \eqref{eq:caffarellisilvestre}, satisfies Hopf's Lemma in the sense that 
\begin{equation}\label{eq:HopfU} \liminf_{t \to 0^+} \frac{U(x_0,t)-U(x_0,0)}{t^{2s}} > 0.\end{equation}
A more precise statement is given in Theorem \ref{prop:valueoftheintegral} for the ``if'' part, and in Theorem \ref{prop:conversehopf} for the ``only if'' part.

Let us first give a precise definition of what it means for a function $u \in H^s_0(\Omega)$ to satisfy Hopf's Lemma.

\begin{definition} \label{defi:hopf}
Let $\Omega \subset \R^N$ be an open set, and let $x_0 \in \partial \Omega$. Suppose that there exists a ball $B_R(y_0) \subset \Omega$ such that $x_0 \in \partial B_R(y_0)$. For every $\beta \in \left(0, \frac{\pi}{2} \right)$, define
\begin{equation} \label{C_b} C_\beta(x_0) := \left\{x \in \Omega \,\bigg|\,\arccos\left(\frac{x-x_0}{|x-x_0|}\cdot \frac{y_0-x_0}{|y_0-x_0|}\right) < \beta\right\},\end{equation}
Let $u \in H^s_0(\Omega)$ be a function. We say that $u$ \emph{satisfies Hopf's Lemma at $x_0$} if, for every $\beta \in \left(0, \frac{\pi}{2} \right)$, it holds
\[ \liminf_{\stackrel{x \to x_0}{x \in C_\beta(x_0)}} \frac{u(x)}{|x-x_0|^s} > 0.\]
\end{definition}

\begin{prop} \label{prop:weakminimumprinciple}
 Let $c \in L^\infty(\Omega)$, and let $g \in H^s(\R^N)$. Suppose that 
	\begin{equation} \label{eq:conditionforhopf}  \quad \int_{\R^N\setminus B_r(y_0)} \frac{g(x)}{|x-y|^{N+2s}}\,dx \geq 0 \quad \text{for every } y \in B_r(y_0),\end{equation}
	where $B_r(y_0) \subset \R^N$ is a ball of radius $r>0$ centered at  $y_0 \in \R^N$. Suppose moreover that $\lambda_1(B_r(y_0)) > \|c\|_\infty$, where $\lambda_1(B_r(y_0))$ denotes the first eigenvalue of the Dirichlet fractional Laplacian in $B_r(y_0)$, and let $w \in H^s(\R^N)$ satisfy 
	\begin{equation}
		\left\{\begin{array}{r c l l}
			(-\Delta)^s w &\geq& c(x)w& \text{in }B_r(y_0), \\ w &= & g & \text{in }\R^N \setminus B_r(y_0).\end{array}\right.
	\end{equation}
	Suppose that $w^-|_{B_r(y_0)} \in H^s_0(B_r(y_0))$. Then $w \geq 0$ in $B_r(y_0)$.
\end{prop}
\begin{proof}
	Set $w=w^+-w^-$. Testing the equation with $w^- |_{B_r(y_0)}$ we obtain
	\begin{align*}
		-\int_{B_r(y_0)} c(x)(w^-)^2 & \leq \int_{\R^N} \int_{\R^N} \frac{(w(x)-w(y))(w^-(x)|_{B_r(y_0)}-w^-(y)|_{B_r(y_0)})}{|x-y|^{N+2s}}\,dx\,dy = \\ & = -\int_{\R^N} \int_{\R^N} \frac{(w^-(x)|_{B_r(y_0)}-w^-(y)|_{B_r(y_0)})^2}{|x-y|^{N+2s}}\,dx\,dy - 2\int_{B_r(y_0)} \int_{B_r(y_0)} \frac{w^+(x)w^-(y)}{|x-y|^{N+2s}}\,dx\,dy \\ & - 2\int_{B_r(y_0)} \int_{\R^N \setminus B_r(y_0)} \frac{w(x)w^-(y)}{|x-y|^{N+2s}}\,dx\,dy \\ & \leq -\lambda_1(B_r(y_0)) \int_{B_r(y_0)}(w^-)^2 - 2 \int_{B_r(y_0)} w^-(y) \left(\int_{\R^N \setminus B_r(y_0)} \frac{g(x)}{|x-y|^{N+2s}}\,dx \right)\,dy,
	\end{align*}
	Thus, by assumption \eqref{eq:conditionforhopf},
	\begin{equation}\label{1} (\lambda_1(B_r(y_0)) - \|c\|_\infty )\int_{B_r(y_0)}(w^-)^2 \leq 0,\end{equation}
	which implies $w^- \equiv 0$ in $B_r(y_0)$.
	
\end{proof}
\begin{thm} \label{prop:valueoftheintegral}
	Let $\Omega \subset \R^N$ be an open set.  Let $c \in L^\infty(\Omega)$, and let $u \in H^s_0(\Omega) \cap C(\R^N)$ satisfy
	\begin{equation}
		\left\{\begin{array}{r c l l}
			(-\Delta)^s u &\geq& c(x)u& \text{in }\Omega, \\ u &= & 0 & \text{in }\R^N \setminus\Omega.\end{array}\right.
	\end{equation}
	Let $x_0 \in \partial \Omega$, and suppose that $\Omega$ satisfies an interior ball condition at $x_0$.	Suppose that there exists a ball $B_\delta(x_0)$ of radius $\delta >0$ centered at $x_0$ such that $u \geq 0$ in $B_\delta(x_0)$. 
	Let $U$
	be the Caffarelli-Silvestre extension of $u$.
	If 
	\begin{equation}\label{HopfsU}  \liminf_{t \to 0^+} \frac{U(x_0,t)-U(x_0,0)}{t^{2s}} > 0, \end{equation}
	then $u$ satisfies Hopf's Lemma at $x_0$ in the sense of Definition \ref{defi:hopf}.
\end{thm}	

\begin{proof}		
	Arguing as in \cite[Section 3.1]{caffarellisilvestre} we obtain, by using the change of variable $z= \left( \frac{t}{2s}\right)^{2s}$,
	\[ \frac{U(x_0,z)-U(x_0,0)}{z} = \tilde{P}_{N,s} \int_{\R^N} \frac{u(x)}{\left(|x_0-x|^2 + 4s^2 |z|^{\frac{1}{s}}\right)^{\frac{N+2s}{2}}}\,dx\]
	for every $z>0$, where $\tilde{P}_{N,s} = (2s)^{2s} P_{n,s}$  is a positive constant depending on $N$ and $s$. Let us decompose the integral into the sum of two terms:
	\begin{align*}  \int_{\R^N} \frac{u(x)}{\left(|x_0-x|^2 + 4s^2 |z|^{\frac{1}{s}}\right)^{\frac{N+2s}{2}}}\,dx  & =  \int_{\R^N \setminus B_\delta(x_0)} \frac{u(x)}{\left(|x_0-x|^2 + 4s^2 |z|^{\frac{1}{s}}\right)^{\frac{N+2s}{2}}}\,dx \\ & +  \int_{B_\delta(x_0)} \frac{u(x)}{\left(|x_0-x|^2 + 4s^2 |z|^{\frac{1}{s}}\right)^{\frac{N+2s}{2}}}\,dx.\end{align*}
	For $z \to 0^+$, the first integral converges to
	\[ \int_{\R^N \setminus B_\delta(x_0)} \frac{u(x)}{|x_0-x|^{N+2s}}\,dx \]
	by Lebesgue's Dominated Convergence Theorem, while the second integral tends to 
	\[ \int_{B_\delta(x_0)} \frac{u(x)}{|x_0-x|^{N+2s}}\,dx \]
	by the Monotone Convergence Theorem. Observe that this last quantity might be equal to $+\infty$.
	By \eqref{HopfsU} we obtain that
	\[ \int_{\R^N} \frac{u(x)}{|x_0-x|^{N+2s}}\,dx  > 0,\]
	where the value of the integral might be equal to $+\infty$.
	By continuity, there exists  $r \in \left(0,\frac{\delta}{2}\right)$ such that
		\[  \int_{\R^N \setminus B_r(y_0)} \frac{u(x)}{|x_0-x|^{N+2s}}\,dx > 0 \]
		and
	\[  \int_{\R^N \setminus B_r(y_0)} \frac{u(x)}{|x-y|^{N+2s}}\,dx > 0 \quad \text{for every } y \in B_r(y_0),\]
	where $B_r(y_0)$ is a ball of radius $r>0$ centered at a point $y_0 \in \Omega$ such that $x_0 \in \partial B_r(y_0)$, and $B_r(y_0) \subset B_\delta(x_0)$. Moreover, by the scaling properties of the eigenvalues of the fractional Laplacian, it is possible to choose $r$ such that, without loss of generality, $\lambda_1(B_r(y_0))>\|c\|_\infty$. 
	Let $\psi \in H^s_0(B_r(y_0))$ be the solution to
	\begin{equation*}
		\left\{\begin{array}{r c l l}
			(-\Delta)^s \psi &=& 1& \text{in }B_r(y_0), \\ \psi &= & 0 & \text{in }\R^N \setminus B_r(y_0).\end{array}\right.
	\end{equation*}
	$\psi$ has the explicit expression
	\[ \psi(x) = \gamma_{N,s}(r^2 - |x-x_0|^2)^s,\qquad \gamma_{N,s}:=\frac{4^{-s}\Gamma\left( \frac{N}{2}\right)}{\Gamma\left( \frac{N}{2}+s\right)\Gamma(1+s)}.\]
	Let $R >0$ be such that there exists a ball $B_{2R}(z_0) \subset \Omega$ of radius $2R$ centered at $z_0 \in \Omega$  such that $\text{dist}(B_{2R}(z_0),B_r(y_0))>0$. For $\alpha > 0$, define
	\[ \eta := \psi + \alpha \xi,\]
	where $\xi \in \mathcal C^\infty_c(B_{2R}(z_0))$ is a nonnegative function, such that $\xi \geq 1$ in $B_R$. Let $\varphi \in H^s_0(B_r(y_0))$ be a nonnegative function. Then we have
	\begin{align*}
		\int_{\R^N} \int_{\R^N} \frac{(\eta(x)-\eta(y))(\varphi(x)-\varphi(y))}{|x-y|^{N+2s}}\,dx\,dy & = 	\int_{\R^N} \int_{\R^N} \frac{(\psi(x)-\psi(y))(\varphi(x)-\varphi(y))}{|x-y|^{N+2s}}\,dx\,dy \\ & + \alpha 	\int_{\R^N} \int_{\R^N} \frac{(\xi(x)-\xi(y))(\varphi(x)-\varphi(y))}{|x-y|^{N+2s}}\,dx\,dy \\ & = \int_{B_r(y_0)} \varphi(x)\,dx - 2\alpha \int_{B_r(y_0)} \int_{B_{2R}(z_0)} \frac{\varphi(x)\xi(y)}{|x-y|^{N+2s}}\,dx\,dy \\ & \leq \int_{B_r(y_0)} \varphi(x)\,dx - 2\alpha \int_{B_r(y_0)} \int_{B_{R}(z_0)} \frac{\varphi(x)}{|x-y|^{N+2s}}\,dx\,dy \\ & \leq (1-\alpha C) \int_{B_r(y_0)} \varphi(x)\,dx.
	\end{align*}
	By taking $\alpha$ sufficiently big so that
\[ 1-\alpha C \leq  -\|c\|_\infty\|\psi\|_{\infty},\]
we obtain that, in $B_r(y_0)$,
 \[ (-\Delta)^s \eta \leq -\|c\|_\infty\|\psi\|_{\infty}\leq -\|c\|_\infty\psi= -\|c\|_\infty\eta\leq c(x) \eta,\]
since $\text{dist}(B_{2R}(z_0),B_r(y_0))>0$ and $\eta = \psi \geq 0$ in $B_r(y_0)$.
	
	By continuity, there exists $\varepsilon > 0$ sufficiently small such that the condition
	\[  \quad \int_{\R^N \setminus B_r(y_0)} \frac{u(x) - \varepsilon \eta(x)}{|x-y|^{N+2s}}\,dx >0 \quad \text{for every } y \in B_r(y_0)\]
	is satisfied. Recalling that $u$ is non-negative in $B_r(y_0)$, we can assert that $(u-\varepsilon \eta)^-|_{B_r(y_0)} \in H^s_0(B_r(y_0))$. By Proposition \ref{prop:weakminimumprinciple} with $w=u-\varepsilon \eta \in H^s_0(\Omega)$ it holds
	\[ u \geq \varepsilon \psi \qquad \text{in }B_r(y_0),\]
	which implies that $u$ satisfies Hopf's Lemma at $x_0 \in \partial \Omega$.
\end{proof}

\begin{rem}
The condition
\begin{equation} \label{eq:hopfcaffsilv}  \liminf_{t \to 0^+} \frac{U(x_0,t)-U(x_0,0)}{t^{2s}} > 0.\end{equation}
is satisfied when, in particular, $u$ is a nontrivial solution to
		\begin{equation}
		\left\{\begin{array}{r c l l}
			(-\Delta)^s u &=& c(x)u& \text{in }\Omega, \\ u &= & 0 & \text{in }\R^N \setminus\Omega,\end{array}\right.
	\end{equation}
and	there exists a neighbourhood $\mathcal{V} \subset \R^{N+1}_+$ with $(x_0,0) \in \mathcal{V}$ such that $U \geq 0$ in $\mathcal{V}$. Indeed, in this case, the strong maximum principle for strictly elliptic operators guarantees that either $U\equiv 0$ or  $U>0$ in $ \mathring{\mathcal{V}}$ (see for example \cite[Remark 4.2]{cabresire}). The former case cannot occur because otherwise by continuity we would have $u=0$ in an open set, a contradiction to the unique continuation property for the fractional Laplacian \cite[Theorem 1.2]{ghoshsalouhlmann}. Since $U(x_0,0)=u(x_0)=0$, the Hopf's Lemma for the Caffarelli-Silvestre extension \cite[Proposition 4.11]{cabresire} implies \eqref{eq:hopfcaffsilv}.
\end{rem}

\begin{thm} \label{prop:conversehopf}
	Let $\Omega \subset \R^N$ be an open set.  Let $c \in L^\infty(\Omega)$, and let $u \in H^s_0(\Omega) \cap C(\R^N)$ be such that $u^- \in L^\infty(\Omega)$. Let $x_0 \in \partial \Omega$. Suppose that $\Omega$ satisfies an interior ball condition at $x_0$, and suppose that there exists a ball $B_{\delta}(x_0)$ such that $u \geq 0$ in $B_{\delta}(x_0) \cap \Omega$. Let $U$ 
	be the Caffarelli-Silvestre extension of $u$.
	If $u$ satisfies Hopf's Lemma at $x_0$ in the sense of Definition \ref{defi:hopf}, then
	\[	 \lim_{t \to 0^+} \frac{U(x_0,t)-U(x_0,0)}{t^{2s}} = +\infty,\]
and, therefore,  $U$  satisfies Hopf's Lemma in the sense of \eqref{eq:HopfU}.
\end{thm}

\begin{proof}
	Fix $\beta \in \left(0,\frac{\pi}{2}\right)$, and set $C_\beta(x_0)$ as in \eqref{C_b}. 
	Since $u$ satisfies Hopf's lemma at $x_0$, we can assert that there exists $r \in (0,\delta)$ and $c>0$ such that
	\[ u(x)\geq c|x-x_0|^s \qquad\forall x\in C_\beta(x_0) \cap B_r(x_0).\]
	Using the representation formula \eqref{eq:representation} for the solution $U$ to \eqref{eq:caffarellisilvestre}, we obtain
	\[U(x,t)=P_{N,s}\int_{\R^n}\frac{u(y)t^{2s}}{(|x-y|^2+t^{2})^\frac{N+2s}{2}}\,dy,
	\]	
which is in particular true for $x=x_0$. Taking into account that $u = 0$ in $\R^N \setminus \Omega$ and $u(x_0)=0$, we consider the difference quotient
	\[  \frac{U(x_0,t)-U(x_0,0)}{t^{2s}} = \frac{U(x_0,t)}{t^{2s}} = P_{N,s}\int_{\Omega}\frac{u(y)}{(|x_0-y|^2+t^{2})^\frac{N+2s}{2}}\,dy.\]
	We now split the domain $\Omega$ as follows:
	\[ \Omega = (B_{r}(x_0)\cap C_\beta(x_0))\cup((\Omega\setminus C_\beta(x_0)) \cap B_{r}(x_0))   \cup (\Omega\setminus B_{r}(x_0)).\]
	In $B_r(x_0)\cap C_\beta(x_0)$, we have
	\[ \int_{B_r(x_0)\cap C_\beta(x_0)} \frac{u(y)}{(|x_0-y|^2+t^{2})^\frac{N+2s}{2}}\,dy \geq c\int_{B_r(x_0)\cap C_\beta(x_0)} \frac{|x_0-y|^{s}}{(|x_0-y|^2+t^{2})^\frac{N+2s}{2}}\,dy. \]
	By the Monotone Convergence Theorem, the last integral tends, up to a multiplicative constant, to
	\[ \int_0^{r} \frac{1}{z^{1+s}}\,dz = +\infty.\]
	In $(\Omega\setminus C_\beta(x_0)) \cap B_{r}(x_0)$, it holds	
	\[ \int_{(\Omega\setminus C_\beta(x_0)) \cap B_{r}(x_0)} \frac{u(y)}{(|x_0-y|^2+t^{2})^\frac{N+2s}{2}}\,dy \geq 0,\]
	and since $u \geq 0$ in $B_r(x_0)$, the Monotone Convergence Theorem implies that the integral converges to a nonnegative quantity (finite or infinite). Finally, on $\Omega\setminus B_{r}(x_0)$, we exploit the boundedness of $u^-$, so that
	\begin{align*} \int_{\Omega\setminus B_{r}(x_0)} \frac{u(y)}{(|x_0-y|^2+t^{2})^\frac{N+2s}{2}}\,dy & \geq - \|u^-\|_{\infty} \int_{\Omega\setminus B_{r}(x_0)} \frac{1}{(|x_0-y|^2+t^{2})^\frac{N+2s}{2}}\,dy \\ & \geq - \|u^-\|_{\infty} \int_{\R^N \setminus B_{r}(x_0)} \frac{1}{(|x_0-y|^2+t^{2})^\frac{N+2s}{2}}\,dy \\ & \geq - \|u^-\|_{\infty} \int_{\R^N \setminus B_{r}(x_0)} \frac{1}{|x_0-y|^{N+2s}}\,dy \\ & = - \|u^-\|_{\infty} (N\omega_N)\int_{r}^{+\infty} \frac{1}{z^{1+2s}}\,dz > -\infty,\end{align*}
	where $\omega_N$ is the volume of the $N$-dimensional unit ball.
	All in all, since $C_{n,s}$ is a positive constant, we finally obtain that
	\[  \lim_{t \to 0^+} \frac{U(x_0,t)-U(x_0,0)}{t^{2s}} = +\infty.\]
\end{proof}

\section{On a result by Dipierro-Soave-Valdinoci}\label{secUu}

In the following we will prove a Hopf's Lemma result in the spirit of \cite[Theorem 1.2]{dipierrosoavevaldinoci} by exploting similar arguments as in the previous section.

\begin{thm}\label{thm:equiv-lim}
Let $\Omega \subset \R^N$ be an open set.  Let $c \in L^\infty(\Omega)$, and let $u \in H^s_0(\Omega) \cap C(\R^N)$ satisfy 
\begin{equation}\label{eq1}
	\left\{\begin{array}{r c l l}
		(-\Delta)^s u &\geq & c(x)u& \text{in }\Omega, \\ u &= & 0 & \text{in }\R^N \setminus\Omega.\end{array}\right.
\end{equation}	
We suppose that $u^- \in L^\infty(\Omega)$. Let $x_0 \in \partial \Omega$. Suppose that $\Omega$ satisfies an interior ball condition at $x_0$, and suppose that there exists a ball $B_{\delta}(x_0)$ such that $u \geq 0$ in $B_{\delta}(x_0) \cap \Omega$. For every $\beta \in \left(0, \frac{\pi}{2} \right)$, let $C_\beta(x_0)$ be as defined in \eqref{C_b}.
For $r>0$ small enough, let $z_r \in \Omega$ such that $B_r(z_r) \subset \Omega$ and $x_0 \in \partial B_r(z_r)$.
The two following conditions are equivalent:
\begin{enumerate}
	\item[(i)]  \[ \limsup_{r \to 0^+} \left(\frac{1}{r^{2s}}\inf_{B_{r/2}(z_r)} |u| \right) > 0.\]
	\item[(ii)] For every $\beta \in \left(0, \frac{\pi}{2} \right)$, \[ \liminf_{\stackrel{x \to x_0}{x \in C_\beta(x_0)}} \frac{u(x)}{|x-x_0|^s} > 0.\]

\end{enumerate}
\end{thm}
\begin{proof}
$(i) \Rightarrow (ii)$. Set
\[ \gamma := \limsup_{r \to 0^+} \left(\frac{1}{r^{2s}}\inf_{B_{r/2}{(z_r)}} |u| \right), \]
which is well defined thanks to the interior ball condition. By assumption, it holds $\gamma>0$. By definition of limsup, there exists a sequence $\{r_n\}_{n \in \N}$ of positive numbers, monotonically decreasing to $0$, such that, if we set $B_n := B_{r_n/2}(z_{r_n})$,
\[ \inf_{x \in B_n} u(x) \geq \frac{\gamma}{2}(r_n)^{2s}.\]
Without loss of generality, we can suppose that the balls $B_n$ are pairwise disjoint. Define
\[ C := \bigcup_{n \in \N} B_n.\]
Let $y \in C$. By definition, there exists $n \in \N$ such that $y \in B_n$. Observe that, by the triangle inequality, 
\[ |x_0-y| \leq |x_0-z_{r_n}| + |z_{r_n}-y| \leq \frac{3}{2}r_n \]
and
\[ |x_0-y| \geq |x_0-z_{r_n}|-|z_{r_n} - y| \geq \frac{1}{2}r_n.\]
As a consequence,
\[ u(y) \geq \inf_{x \in B_n} u(x) \geq \frac{\gamma}{2}(r_n)^{2s} \geq \tilde{\gamma} |x_0-y|^{2s},\]
where $\tilde{\gamma}>0$ is independent on $n$.
Decompose
\[ \Omega = C \cup ((B_{\delta}(x_0)\cap \Omega) \setminus C) \cup (\Omega \setminus (C \cup B_{\delta}(x_0)).\]
We have that
\[ \int_{B_{\delta}(x_0) \setminus C} \frac{u(y)}{|x_0-y|^{N+2s}}\,dy \geq 0,\]
the value of the integral being possibly $+\infty$. Moreover,
\begin{align*} \int_{\Omega \setminus (C \cup B_{\delta}(x_0))} \frac{u(y)}{|x_0-y|^{N+2s}}\,dy  & \geq - \|u^-\|_{\infty} \int_{\Omega \setminus (C \cup B_{\delta}(x_0))} \frac{1}{|x_0-y|^{N+2s}}\,dy \\ & \geq - \|u^-\|_{\infty} \int_{\R^N \setminus B_{\delta}(x_0)} \frac{1}{|x_0-y|^{N+2s}}\,dy \\ & = - \|u^-\|_{\infty} (N\omega_N)\int_{\delta}^{+\infty} \frac{1}{z^{1+2s}}\,dz > -\infty,\end{align*}
where $\omega_N$ is the volume of the $N$-dimensional unit ball. Furthermore,
\[ \int_{C} \frac{u(y)}{|x_0-y|^{N+2s}}\,dy \geq \tilde{\gamma} \int_{C} \frac{1}{|x_0-y|^N}\,dy = \sum_{n=0}^{+\infty} \int_{B_n} \frac{1}{|x_0-y|^N}\,dy \geq \overline{c} \sum_{n =0}^{+\infty} \frac{(r_n)^N}{(r_n)^N} = +\infty.\]
Summing up, we obtain
\[ \int_{\Omega} \frac{u(y)}{|x_0-y|^{N+2s}}\,dy = + \infty.\]
The conclusion then follows by arguing as in Theorem \ref{prop:valueoftheintegral}.

$(ii) \Rightarrow (i)$. Suppose by contradiction that
\[ \limsup_{r \to 0^+} \left(\frac{1}{r^{2s}}\inf_{B_{r/2}(z_r)} u \right) = 0.\]
This entails the existence of a sequence $\{r_n\}_{n \in \N}$ of positive numbers, monotonically decreasing to $0$, and of points $\{x_n\}_{n \in \N}$ such that $x_n \in B_{r_n/2}(z_{r_n})$ for every $n \in \N$, and such that
\[ \frac{u(x_n)}{(r_n)^{2s}} \to 0\]
as $n \to +\infty$. By construction, the points $x_n$ also belong to $C_\beta(x_0)$ for $\beta = \frac{\pi}{6}$. Moreover, it holds
\[ |x_n-x_0| \geq |x_0-z_{r_n}|-|z_{r_n}-x_n| \geq \frac{r_n}{2},\]
so that
\[ \lim_{n \to +\infty} \frac{u(x_n)}{|x_n-x_0|^{2s}}= 0,\]
which implies
\[ \lim_{n \to +\infty} \frac{u(x_n)}{|x_n-x_0|^{s}}= 0,\]
a contradiction to the assumption
\[ \liminf_{\stackrel{x \to x_0}{x \in C_\beta(x_0)}} \frac{u(x)}{|x-x_0|^s} > 0.\]


\end{proof}
\begin{rem}We emphasize here that condition $(i)$ in Theorem \ref{thm:equiv-lim} implies the fast growth property used in \cite{dipierrosoavevaldinoci} to characterize the validity of Hopf's Lemma, i.e., 
	\[ \limsup_{r \to 0^+} \left(\frac{1}{r^{2s}}\inf_{B_{r/2}(z_r)} |u| \right) =+\infty.\] We only need to observe that the case \begin{equation} \label{eq:improvementvaldinoci} \limsup_{r \to 0^+} \left(\frac{1}{r^{2s}}\inf_{B_{r/2}(z_r)} |u| \right) \in (0,+\infty),\end{equation}
	cannot occur.	Indeed, if that were possible, on the one hand, reasoning as in the proof of the implication $(i) \Rightarrow (ii)$ in Theorem \ref{thm:equiv-lim} we would get the validity of the Hopf's Lemma. But, on the other hand, the same contradiction argument used for the opposite implication $(ii) \Rightarrow (i)$ implies that Hopf's Lemma does not hold true when the condition \eqref{eq:improvementvaldinoci} is satisfied. 
	Therefore, similarly to \cite[p. 222]{dipierrosoavevaldinoci}, one can state the following slightly improved dichotomy:
	 \begin{itemize}
	 	\item either $\displaystyle \limsup_{r \to 0^+} \left(\frac{1}{r^{2s}}\inf_{B_{r/2}(z_r)} |u| \right) = + \infty$, and $u$ satisfies Hopf's Lemma;
	 	\item or $\displaystyle \limsup_{r \to 0^+} \left(\frac{1}{r^{2s}}\inf_{B_{r/2}(z_r)} |u| \right) = 0$, and $u$ does not satisfy Hopf's Lemma.
	 \end{itemize}
\end{rem}
\subsection*{Acknowledgements} 

The authors would like to express their sincere gratitude to Serena Dipierro and Enrico Valdinoci for their invaluable suggestions and insightful conversations during their visit to Sapienza, Università di Roma in July 2024. They also thank the anonymous referees for their valuable comments.

A. DlT. acknowledges financial support from the Spanish Ministry of Science and Innovation (MICINN), through the
IMAG-Maria de Maeztu Excellence Grant CEX2020-001105-M/AEI/10.13039/501100011033. She is also supported by the FEDER-MINECO Grants
PID2021- 122122NB-I00, PID2020-113596GB-I00 and PID2024-155314NB-I00; RED2022-134784-T, funded by MCIN/AEI/10.13039/501100011033
and by J. Andalucia (FQM-116); Fondi Ateneo - Sapienza Università di Roma; PRIN (Prot. 20227HX33Z); DFG - Projektnummer: 561401741; and INdAM-GNAMPA Project 2023, codice CUP E53C2200193000, INdAM-GNAMPA  Project 2024, codice
CUP E53C23001670001 and INdAM-GNAMPA  Project 2025, codice CUP E5324001950001.

E. P. acknowledges partial support from the ANR project STOIQUES (Shape and Topology Optimization: Impactful QUestions and Emerging Subjects), ANR-24-CE40-2216.

\end{document}